\documentclass[letter,11pt]{amsart}
\usepackage[toc,page]{appendix}
\usepackage{amssymb} 
\usepackage{bbm}
\usepackage{graphicx}
\usepackage{color}
\usepackage[font=footnotesize]{caption}
\usepackage{mathtools}
\usepackage{pinlabel}
\usepackage{graphicx}
\graphicspath{ {images/} }
\usepackage{subfig}
\usepackage{comment}
\usepackage{hyperref}

\usepackage{float}
\usepackage{calc}
\def\thetitle{{ . }}
\hypersetup{
	pdftitle=  \thetitle,
	pdfauthor=  {}  %Author name
}

\newtheorem{thm}{Theorem}[section]
\newtheorem*{thm*}{Theorem}
\newtheorem{lem}[thm]{Lemma}
\newtheorem{cor}[thm]{Corollary}
\newtheorem{prop}[thm]{Proposition}
\newtheorem{que}{Question}

\newtheorem{defn}[thm]{Definition}

\theoremstyle{remark}

\newtheorem*{rmk}{\textbf{Remark}}

\usepackage{tikz}
\usetikzlibrary{calc,decorations.markings}
\usetikzlibrary{shapes,snakes}

\theoremstyle{definition}
\newtheorem*{defn*}{Definition}

\newcommand\N{\mathbb{N}}

\newcommand\R{\mathbb{R}}

\newcommand\Mod{\operatorname{Mod}}
\newcommand\vol{\operatorname{vol}}

\newcommand\PSL{\operatorname{PSL}}

\title[Topological entropy and homology of mapping tori]{Topological entropy of pseudo-Anosov maps on punctured surfaces vs. homology of mapping tori}

\author{Hyungryul Baik}
\address{Department of Mathematical Sciences, KAIST,  
	291 Daehak-ro, Yuseong-gu, Daejeon 34141, South Korea }
\email{hrbaik@kaist.ac.kr}

\author{Juhun Baik}
\address{Department of Mathematical Sciences, KAIST,  
	291 Daehak-ro, Yuseong-gu, Daejeon 34141, South Korea }
\email{jhbaik@kaist.ac.kr}

\author{Changsub Kim}
\address{Department of Mathematical Sciences, KAIST,  
	291 Daehak-ro, Yuseong-gu, Daejeon 34141, South Korea }
\email{kcs55505@kaist.ac.kr}

\author{Philippe Tranchida}
\address{Department of Mathematical Sciences, KAIST,  
	291 Daehak-ro, Yuseong-gu, Daejeon 34141, South Korea }
\email{ptranchi@kaist.ac.kr}

\date{\today}

\begin{document}
\maketitle

\begin{abstract}
	We investigate the relation between the topological entropy of pseudo-Anosov maps on surfaces with punctures and the rank of the first homology of their mapping tori. On the surface $S$ of genus $g$ with $n$ punctures, we show that the entropy of a pseudo-Anosov map is bounded from above by $\dfrac{(k+1)\log(k+3)}{|\chi(S)|}$ up to a constant multiple when the rank of the first homology of the mapping torus is $k+1$ and $k, g, n$ satisfy a certain assumption. This is a partial generalization of precedent works of Tsai and Agol-Leininger-Margalit. 
\end{abstract}

%\keywords{}

\section{Introduction}\label{sec:intro}

Let $S$ be a connected orientable surface of finite type and $\Mod(S)$ be its mapping class group. 
Denote by $S_{g,n}$ a surface of genus $g$ with $n$ punctures. 
It is well-known that when $f$ is a pseudo-Anosov mapping class, the pseudo-Anosov map representative minimizes the topological entropy among all other representatives of the isotopy class $f$. For this reason, we do not distinguish a pseudo-Anosov mapping class and its pseudo-Anosov representative, and there is no ambiguity to discuss their topological entropy. Also, we simply talk about entropy, since it always mean topological entropy throughout the paper.   	
	
	For a pseudo-Anosov $f \in \Mod(S_{g,n})$, let 
	\begin{enumerate}
		\item $\kappa(f)$ denote the dimension of the subspace of $H_1(S_{g,n}; \mathbb{R})$ fixed by $f$, and
		\item $h(f)$ denote the entropy of $f$.
	\end{enumerate}

We remark that $\kappa(f)$ is an integer between $0$ and $2g+n$ and that $h(f)$ is equal to the logarithm of the stretch factor $\lambda(f)$ of $f$. For more details on the relation between $h(f)$ and $\lambda(f)$, we refer to \cite{FLP}.
Somewhat surprisingly, $\kappa(f)$ and $h(f)$ are related in an interesting way. To see this more explicitly, let $L(k,g)$ be the $\inf\{h(f) \mid f\colon S_g \to S_g$ is pseudo-Anosov and $\kappa(f) \geq k\}$. Then Agol-Leininger-Margalit \cite{agol2016pseudo} showed that 
\[
    L(k, g) \asymp \dfrac{k+1}{g} (\asymp \dfrac{k+1}{|\chi(S_g)|}), 
\]
where the symbol $\asymp$ means that the ratio between the left-hand side and the right-hand side is uniformly bounded from above and below. It was known earlier that $L(0, g) \asymp 1/g$ by Penner \cite{penner1991bounds} and $L(2g, g) \asymp 1$ by Farb-Leininger-Margalit \cite{farb2008lower}, and the result of Agol-Leininger-Margalit interpolates this two extreme cases and shows a mysterious connection between the dimension of the invariant homology and minimal entropy (in particiular answering a question of Ellenberg \cite{ellenberg2010pseudo}). 

One can ask whether we can generalize the result of Agol-Leininger-Margalit to punctured surfaces. 
For a pseudo-Anosov map $f \in \Mod(S_{g,n})$, the quantities $\kappa(f)$ and $h(f)$ are defined similarly. Analogous to $L(k, g)$, we define $L(k,g,n)$ to be the $\inf\{h(f) \mid f\colon S_{g,n} \to S_{g,n}$ is pseudo-Anosov and $\kappa(f) \geq k\}$.

There have been previous work in the setting of punctured surfaces in the case where $ k=0$. For example, 
Valdivia \cite{valdivia2012sequences} showed that when $g=rn$ for some fixed positive rational $r$, $L(0,g,n)\asymp\dfrac{1}{|\chi(S_{g, n})|}$ .
For fixed $g \geq 2$ Tsai showed that $L(0,g,n)\asymp\dfrac{\log n}{n}$ \cite{tsai2009asymptotic}.
Yazdi \cite{yazdi2020punture} showed that there exists a constant $C$ such that if $g>Cn\log^2 (n)$, then $L(0,g,n)\asymp\dfrac{1}{g}$. In the same article it is also proven that for fixed $n$,we have that  $L(0,g,n)\asymp\dfrac{1}{g}$.
Cooper-Tillmann-Worden \cite{cooper2021thurston} showed that for any fixed $0<\epsilon<1$, for all $g$ and $n$ such that $ \epsilon g+2\leq n \leq \frac{1}{\epsilon}g+2$, we have that $L(0,g,n)\lesssim \dfrac{1}{|\chi(S_{g, n})|}$.

Our main result is the following 
	\begin{thm*} \label{thm:main_thm_intro}
		For $4g-4 \leq k+1 \leq n \leq 2k-4g+6$, we have
		\[
		\dfrac{k+1}{|\chi(S_{g, n})|} \lesssim L(k, g, n)\lesssim \frac{(k+1)\log{(k+3)}}{|\chi(S_{g, n})|} 
		\]
	\end{thm*}
	
\noindent where $\lesssim$ means the inequality holds up to a constant multiple. The theorem follows from the discussion in Section \ref{subsec:lowerbdd} (lower bound) and Theorem \ref{thm:main_ineq} (upper bound). In fact, as discussed in Section \ref{subsec:lowerbdd}, the lower bound holds without any restriction on $k, g, n$ and it follows from some basic 3-dimensional geometry and the Kojima-McShane's inequality $h(f) \gtrsim \dfrac{\vol(M_f)}{|\chi(S)|}$ where $f \in \Mod(S)$ is pseudo-Anosov and $M_f$ is the mapping torus with monodromy $f$ (see Theorem 1.1 of \cite{kojima2018normalized}). 

\begin{figure}[h]
    \centering
    \includegraphics[width = 0.7\textwidth]{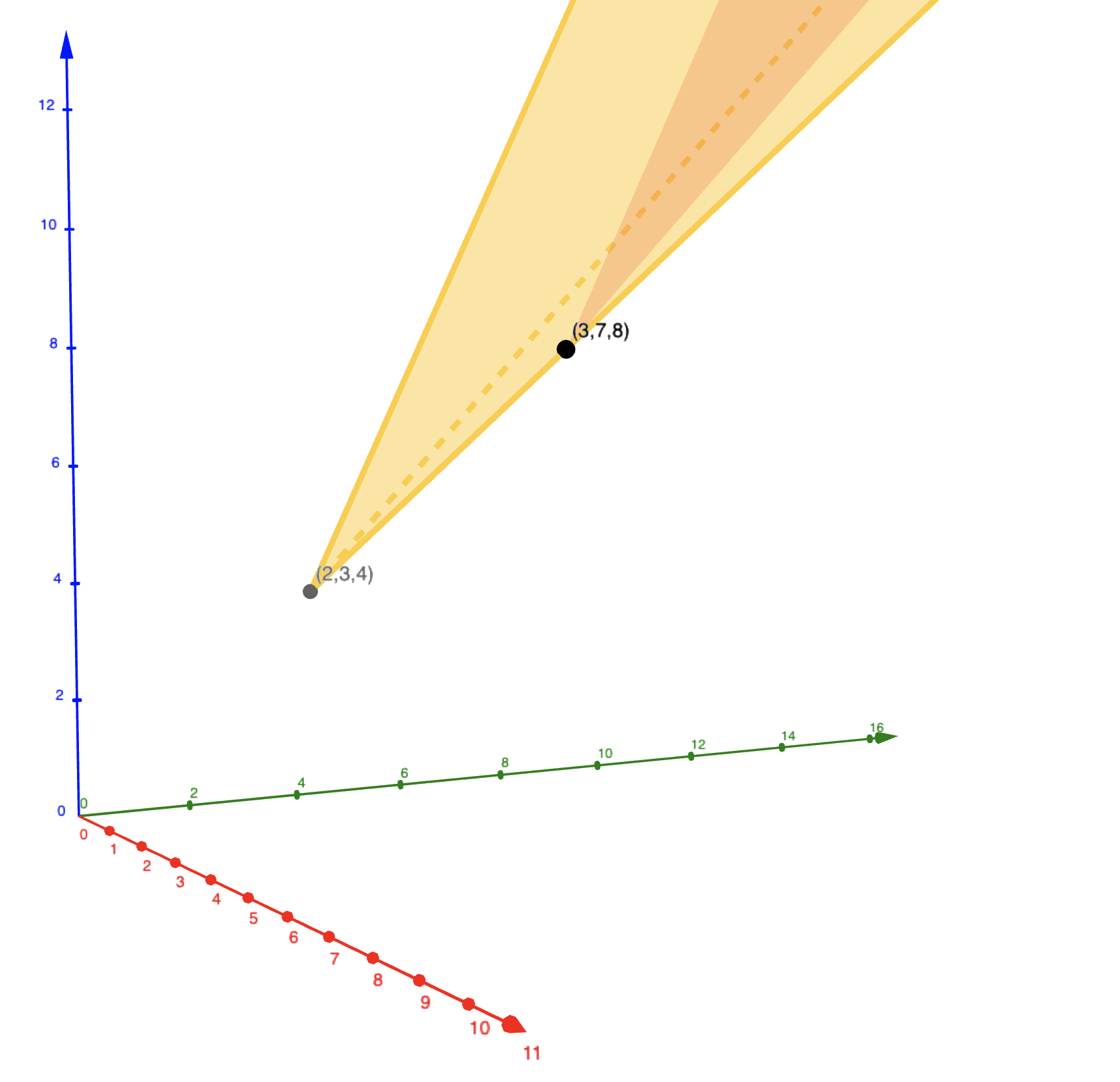}
    \caption{A visualisation of the domain for the main inequality of Theorem \ref{thm:main_thm_intro}. The variables $g, k$ and $n$ correspond to the red, green and blue axis, respectively.
    The darker region is the slice of the cone at $g=3$
    }
    \label{fig:plot}
\end{figure}

To obtain the upper bound in the main theorem, we investigate certain sequences of fibers in the fibered cone of hyperbolic link complements which fiber over the circle. Let $C(n)$ be a $n$ chained link with $2$ half-twists (for a more precise description, see Section \ref{sec:nchainlink}), and let $M(n)$ be the complement of a small regular neighborhood of $C(n)$. We first study a specific fibration of $M(n)$ in Section \ref{sec:stretch_factor}, and then study primitive integral classes in the fibered cone which are projectively near the given fibration . As in \cite{agol2016pseudo}, the fact that the normalized entropy of monodromies extends to a continuous convex function on the fibered face (Fried \cite{fried1982flowequivalence}, \cite{fried1983anosovflow}) then allows us to compute this upper bound (the (non-)convexity of translation lengths on curve complexes and arc complexes are also a topic of active research. See \cite{baik2019asymptotic}, \cite{strenner2018fibrations} for example). In Section \ref{sec:magic3mfld}, we describe the situation in a toy example, namely the magic manifold. 

\begin{figure}
    \centering
    \includegraphics[width = 0.7\textwidth]{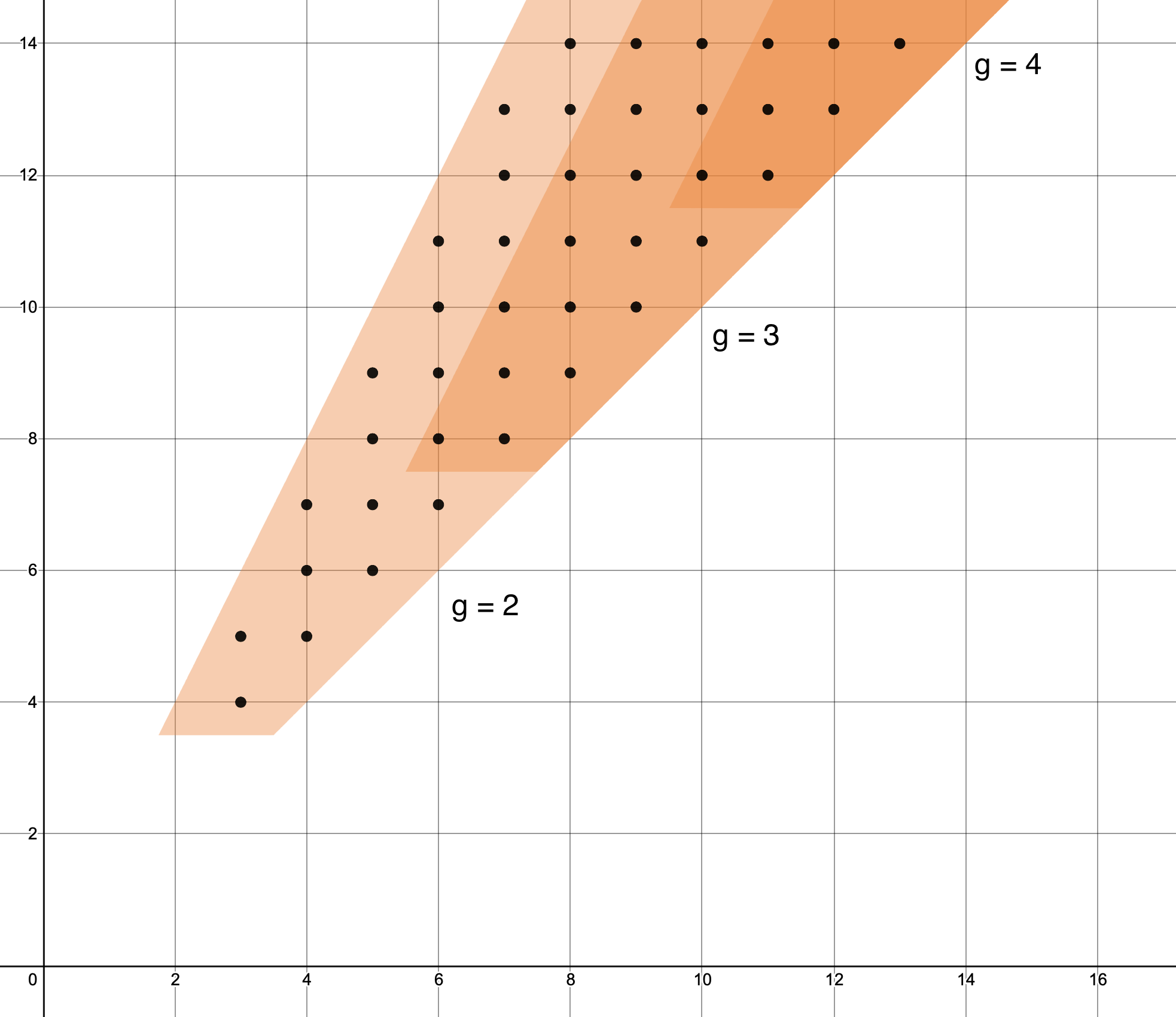}
    \caption{This is a graphical representation of the conditions of the Corollary. The vertical and horizontal axis correspond to $n,k$, respectively. The entire colored area is the $g=2$ slice of our domain and darker areas correspond to slices at higher values of $g$.}
    \label{fig:slice}
\end{figure}

One corollary of the upper bound in the main theorem is that for fixed $g \geq 2$, when $n \geq 4g-4$ and $\dfrac{n}{2} + 2g - 3 < k < n$, we have $L(k, g, n) \lesssim \dfrac{(k+1)\log n}{n}$.
We provide a visual representation of the domains of such slice of fixed genus in figure \ref{fig:slice}.
This is a partial generalization of Tsai's work \cite{tsai2009asymptotic} where it was shown that $L(0, g, n) \asymp \dfrac{\log n}{n}$ for fixed $g \geq 2$. From the viewpoint of \cite{tsai2009asymptotic}, the log term in the upper bound is quite natural. On the other hand, it can be shown that the current method for obtaining the lower bound using Kojima-McShane's inequality cannot be improved to get a better lower bound. But there might be another way to obtain a better lower bound. Also, it would be interesting to explore this question outside the range of $k, g, n$ we covered in our result. We leave this as the following question. 

 \begin{que}
 		For fixed $g$, do we have $L(k, g, n)\asymp \dfrac{(k+1)\log{n}}{|\chi(S_{g, n})|}$ in general ?
 \end{que}\label{que:main}
	Another direction of further research is to study the complements of the $n$ chained links with $p$ half-twists. In an upcoming paper of the second and the fourth authors, they compute the unit Thurston norm balls and Teichm\"uller polynomials in these settings.

	\subsection*{Acknowledgement}

	We thank Robert Billet, Dongryul M. Kim, Livio Liechti, Dan Margalit, JungHwan Park, Chenxi Wu for helpful discussions. We also thank William Worden and his program Tnorm which helped us to understand the shape of some fibered faces.
	All authors were partially supported by Samsung Science and Technology Foundation under Project Number SSTF-BA1702-01. 

\section{Preliminaries and the lower bound}\label{sec:prelim}

As stated in the introduction, we are interested in the asymptotic behavior of the quantity $L(k,g,n) = min\{h(f) \mid f\colon S_{g,n} \to S_{g,n}$ and $\kappa(f) \geq k\}$. 
This quantity interpolates between $L(0, g, n)$, which is the minimal entropy among all pseudo-Anosov map of $S_{g, n}$, and $L(2g, g, 0)$, which is the minimal entropy among all pseudo-Anosov map in a Torelli subgroup. 
In \cite{agol2016pseudo}, the authors establish the asymptotic behavior of $L(k,g)$. We start this section by recalling their arguments for the lower bound and stating them explicitly for the case punctured surfaces.
Then we will quickly review the core notions needed for the rest of paper, such as Thurston's norm and Thurston's construction on surfaces.

\subsection{Lower bound}\label{subsec:lowerbdd}

To get a lower bound for $L(k,g,n)$, we use the same sequence of inequalities as in \cite{agol2016pseudo}. Their argument actually works verbatim for punctured surfaces too, but since they stated it only in the case of closed surfaces we recall it explicitly here.

\begin{thm}[Theorem 1.1 of \cite{kojima2018normalized}]
    For a pseudo-Anosov map $f$ on surface $S_{g,n}$, the inequality
    $$3\pi |\chi(S_{g,n})| h(f) \geq\vol(M_f)  $$
    holds.
\end{thm}

\begin{prop}[Proposition 2.2 of \cite{agol2016pseudo}]
    If $M$ is a complete, orientable, hyperbolic 3-manifold of finite volume, then
    $$b_1(M) \leq 334.08 \cdot \vol(M) $$
\end{prop}
    
Whenever $M$ is a mapping torus with fibration $\varphi$, the first Betti number of $M$ is exactly one more than $\kappa(\varphi)$, the dimension of the homological  fixed subspace of $\varphi$. 
This fact is a direct consequence of the Mayer-Vietoris sequence associated to a mapping torus.
Combining all of this, we get the desired result:

\begin{thm}
    For any values of $k,g$ and $n$, the following holds
    $$.00031\frac{k+1}{|\chi(S_{g,n})|} \leq L(k,g,n)$$
    In other words,
    $$\frac{k+1}{|\chi(S_{g,n})|} \lesssim L(k,g,n)$$
\end{thm}
\begin{proof}
    For any pseudo-Anosov $f$ with $\kappa (f)\geq k$, we get the inequality
    \begin{align*}
	k+1 &\leq \kappa(f) +1\\
	&= b_1(M_{f}) \\
	&\leq 334.08 \cdot \vol(M_f) \\
	&\leq  334.08 \cdot 3\pi |\chi(S_{g,n})| h(f) \\
	&\simeq 3147.8 |\chi(S_{g,n})| h(f).
    \end{align*}
\end{proof}

\subsection{Thurston's norm and fibered faces}\label{subsec:norm_and_fibface}

Let $M$ be an irreducible, atoroidal and oriented $3$-manifold with boundary $\partial M$. 
In \cite{thurston1986norm}, Thurston defined a norm $x$ on the second homology group $H_2(M,\partial M;\mathbb{Z)}$. 
When $a \in H_2(M,\partial M;\mathbb{Z)}$ is fibration $M \to S^1$ and $S$ is fiber, the norm is given by
\[
x(a) = {-\chi(S)}
\]

This function is then extended in the obvious way to rational number and Thurston proved that there is a unique continuous extension to the real numbers which is linear on rays through the origin. 
Thurston also showed that unit ball $B$ with respect to this norm is a convex polyhedron $P$. 
In other words, if $\mathcal{F}$ is a facet of $P$ and $\mathcal{C}$ is the cone over $\mathcal{F}$, the restriction of norm $x$ on $\mathcal{C}$ is linear.

Now suppose that $M \to S^1$ is a fibration of $M$ over the circle with fiber $S_g^n$ and monodromy $\varphi$. 
The surface $S_{g,n}$ then represents an integral point $a \in H_2(M,\partial M;\mathbb{Z)}$. 
Thurston proved that there is a top dimensional face $\Delta$ of $B$ such that $a$ is in the open cone $C(\Delta)$ through the origin and any other integral points $b \in H_2(M,\partial M;\mathbb{Z)}$ have minimal representatives $F_b$ that correspond to fibers of $M$ over $S^1$. 
Such a top dimensional face $\Delta$ is called a \textit{fibered face}. 

\subsection{Thurston's construction}\label{subsec:Thurstonconstruction}
In \cite{thurston1988construction}, Thurston described a way to create mapping classes on a surface and a criterion to check whether the constructed map is pseudo-Anosov. This theorem is now known under the name of Thurston's construction. We will use this to calculate the stretch factors of the monodromies we construct in section \ref{sec:stretch_factor}.

For a simple closed curve $\alpha$, we denote by $T_\alpha$ the positive Dehn twist around the curve $\alpha$. For a  multicurve $A=\{\alpha_1,\cdots,\alpha_m\}$, let us then define $T_A:=T_{\alpha_1}\cdots T_{\alpha_m}$.
Let $S$ be a surface and let $A$ and $B$ be two multicurves on S. 
We say that the two multicurves $A$ and $B$ fill the surface $S$ if, after cutting $S$ along $A$ and $B$, all remaining connected components are homeomorphic to disks or once punctured disks.

\begin{thm}[\cite{thurston1988construction}, Chapter 14.1 in \cite{farb2012primer}] \label{Thurston's construction}
    Let $A=\{\alpha_1,\cdots,\alpha_m\}$ and $B=\{\beta_1,\cdots, \beta_n\}$ be two multicurves that fill a surface $S$. 
    Define $N$ to be the matrix given by $N_{i,j}=i(\alpha_i,\beta_j)$ and $\mu$ to be the largest real eigenvalue of $NN^T$. Then there is a representation $\rho :\langle T_A,T_B\rangle  \to  \PSL{(2,\R)}$ given by
    \[
        T_A \mapsto 
        \begin{bmatrix} 
        1 & \sqrt{\mu} 
        \\ 0 & 1	
        \end{bmatrix} \quad \textrm{and} \quad T_B \mapsto 
        \begin{bmatrix} 	
        1 & 0 \\
        -\sqrt{\mu} & 1 	
        \end{bmatrix}.
    \]
     Furthermore, an element in $\langle T_A,T_B \rangle$ is pseudo-Anosov if and only if its representation is hyperbolic and its stretch factor is then equal to the largest eigenvalue of its representation.
\end{thm}

We remark that the existence of such a $\mu$ is guaranteed by Perron-Frobenious's theorem. For more details on Thurston's construction, we refer to Chapter 14.1 in \cite{farb2012primer}, for example.

\section{Examples in Magic $3$-manifold}\label{sec:magic3mfld}

To illustrate the techniques we will use in this paper, we start with some motivating toy examples which arise from fibers in the magic $3$-manifold. That being said, these examples yield some results which are not covered by the main theorem of this paper, so they are of interest on their own also.

Let $M$ be the magic 3-manifold, which is the complement of a small regular neighborhood of the 3-chain link shown in figure \ref{fig:magic3mfld}.

\begin{figure}[h]
	\centering
	\includegraphics[scale=0.7]{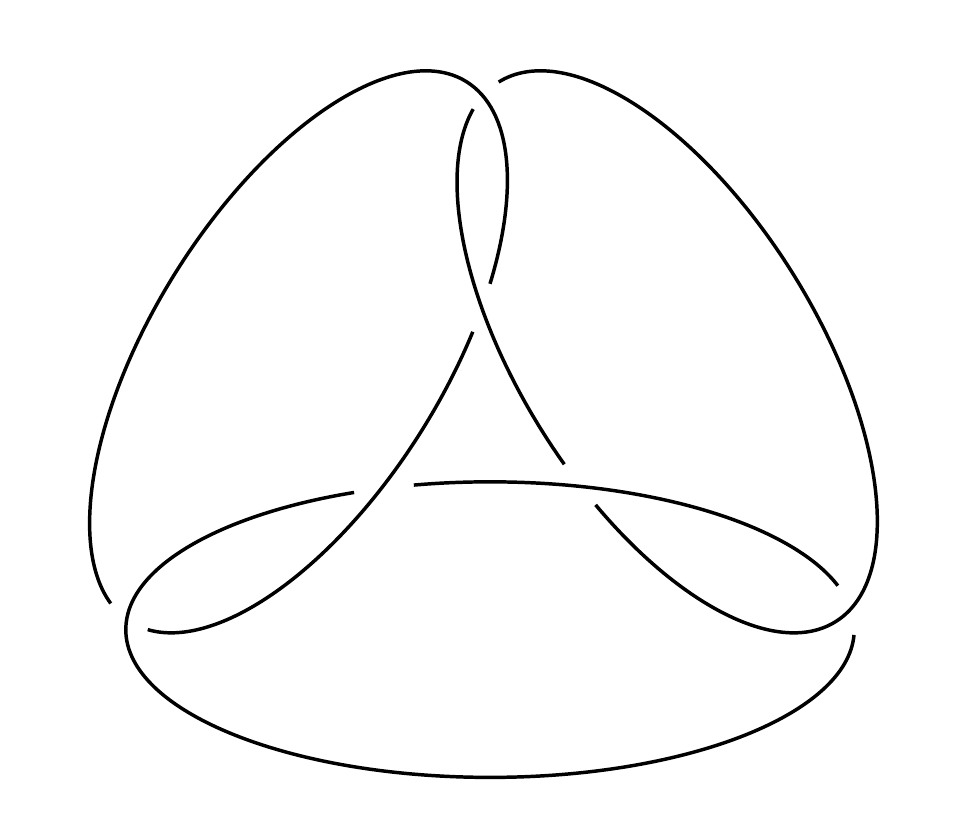}
	\caption{The magic manifold is the complement of a regular neighborhood of this alternating link}
	\label{fig:magic3mfld}
\end{figure}

As the first Betti number of $M$ is $3$, every monodromy map of its fiber fixes a homological subspace of dimension $2$. 
One of the fibered faces $\Delta$ of $M$ is the convex hull of the points $(1, 0, 0)$, $(1, 1, 1)$, $(0, 1, 0)$ and $(0, 0, -1)$. It is described as $\Delta := \{(x, y, z) ~|~ x+y-z = 1, x>0, y>0, x>z, y>z\}$.
In \cite{kin2014dynamics}, Eiko Kin shows $a = (1,1,0)$ represents a fiber and how to calculate the number of boundaries of fibered classes in the fibered cone $C(\Delta) := \R^+\cdot\Delta$ corresponding to the fibered face $\Delta$:

\begin{lem}[Lemma 2.6 in \cite{kin2014dynamics}]\label{lem:bdd_formula}
	Let $(x,y,z)$ be a primitive integer tuple in $C(\Delta)$.
	Then we have
	\begin{enumerate}
		\item $\vert\vert a \vert\vert = x+y-z$
		\item the number of boundaries of $(x,y,z)$ is equal to $\gcd(x, y+z) + \gcd(y, z+x) + \gcd(z, x+y)$.
		(Here we use the convention that $\gcd(0,w) = |w|$)
	\end{enumerate}
\end{lem}

By using the above lemma we can determine the topological type of each fibered class just by knowing its coordinates in $H_2(M,\partial M, \mathbb{Z})$.
Note that the normalized entropy is constant on each ray in a fibered face.  % refer thm/lem in Preliminary part
Also, the entropy function restricted to a fibered face is strictly convex and diverges to $\infty$ towards the boundary.
Using these information, we compute upper bounds for $L(k,g,n)$ in some special cases.
\begin{lem}
	$L(2, g, 3) \asymp 1/g \asymp L(2, g, 4)$. 
\end{lem}
\begin{proof}
	We choose a point near the ray $R$ which passes through $(4,4,2)$ and the origin. 
	It obviously passes through the fibered face $\Delta := \{(x, y, z) ~|~ x+y-z = 1, x>0, y>0, x>z, y>z\}$.
	Now, to get surfaces with $3$ boundary components, we choose the following sequences.
	\begin{itemize}
		\item $(4k+1, 4k+1, 2k+1)$ has $3$ boundaries and the number of genera is $3k$.
		\item $(4k, 4k, 2k+1)$ has $3$ boundaries and the number of genera is $3k-1$.
		\item $(4k-1, 4k-1, 2k+1)$ has $3$ boundaries and the number of genera is $3k-2$ if $k \neq 1 \mod 3$.
		\item $(4k+1, 4k+1, 2k-1)$ has $3$ boundaries and the number of genera is $3k+1$ if $k \neq 2 \mod 3$.
	\end{itemize}
	These sequences cover all cases for $g > 1$ and $n = 3$.
	For surfaces with $4$ boundaries, we choose the following sequences instead.
	\begin{itemize}
		\item $(4k+1, 4k+1, 2k)$ has $4$ boundaries and the number of genera is $3k$.
		\item $(4k-1, 4k-1, 2k)$ has $4$ boundaries and the number of genera is $3k-2$.
		\item $(4k+3, 4k+3, 2k)$ has $4$ boundaries and the number of genera is $3k+2$ if $k \neq 0 \mod 3$.
		\item $(4k-3, 4k-3, 2k)$ has $4$ boundaries and the number of genera is $3k-4$ if $k \neq 0 \mod 3$.
	\end{itemize}
	These sequences cover all cases for $g $ and $n = 4 $.
	Since the ray $R$ is in the interior of the fibered face and all the sequences of points we chose converge to $R$, the normalized entropy of such tuples approaches the one of the ray and is hence bounded.
	Since this upper bound is also a lower bound by the content of section \ref{subsec:lowerbdd}, we get
	$L(2,g,3)\asymp 1/g \asymp L(2, g, 4) $.
	
\end{proof}

\begin{lem}
	$L(2,0,n) \asymp 1/n$.
\end{lem}
\begin{proof}
	We choose the following sequences, so that 
	they all lie around the ray which passes $(1,1,0)$ from the origin.
	\begin{itemize}
		\item $(2k-1, 2k, 0)$ has $4k+1$ boundaries and no genus.
		\item $(2k-1, 2k+1, 0)$ has $4k+2$ boundaries and no genus.
		\item $(2k, 2k+1, 0)$ has $4k+3$ boundaries and no genus.
		\item $(2k-1, 2k+3, 0)$ has $4k+4$ boundaries and no genus.
	\end{itemize}
	These sequences cover all $n>4$. Here again upper bound from sequences and lower bound from section \ref{subsec:lowerbdd} coincide. Thus $L(2,0,n) \asymp 1/n$.
\end{proof}

\section{The $n$ chained link complements and their fibered face}\label{sec:nchainlink}
In this section we generalize the techniques used for the magic $3$-manifold to study sequences of fibers in more general link complements. 
We will concentrate on specific $n$-chained link complement with $2$ half-twists, denoted by $C(n)$.
We will first show that $C(n)$ is fibered and find a fibered face $\mathcal{F}$. Then we will compute the stretch factor of a specific fiber lying in $\mathcal{F}$.

\subsection{The $n$ chained links $C(n)$ are fibered}
\begin{figure}
    \centering
    \includegraphics[width=.5\textwidth]{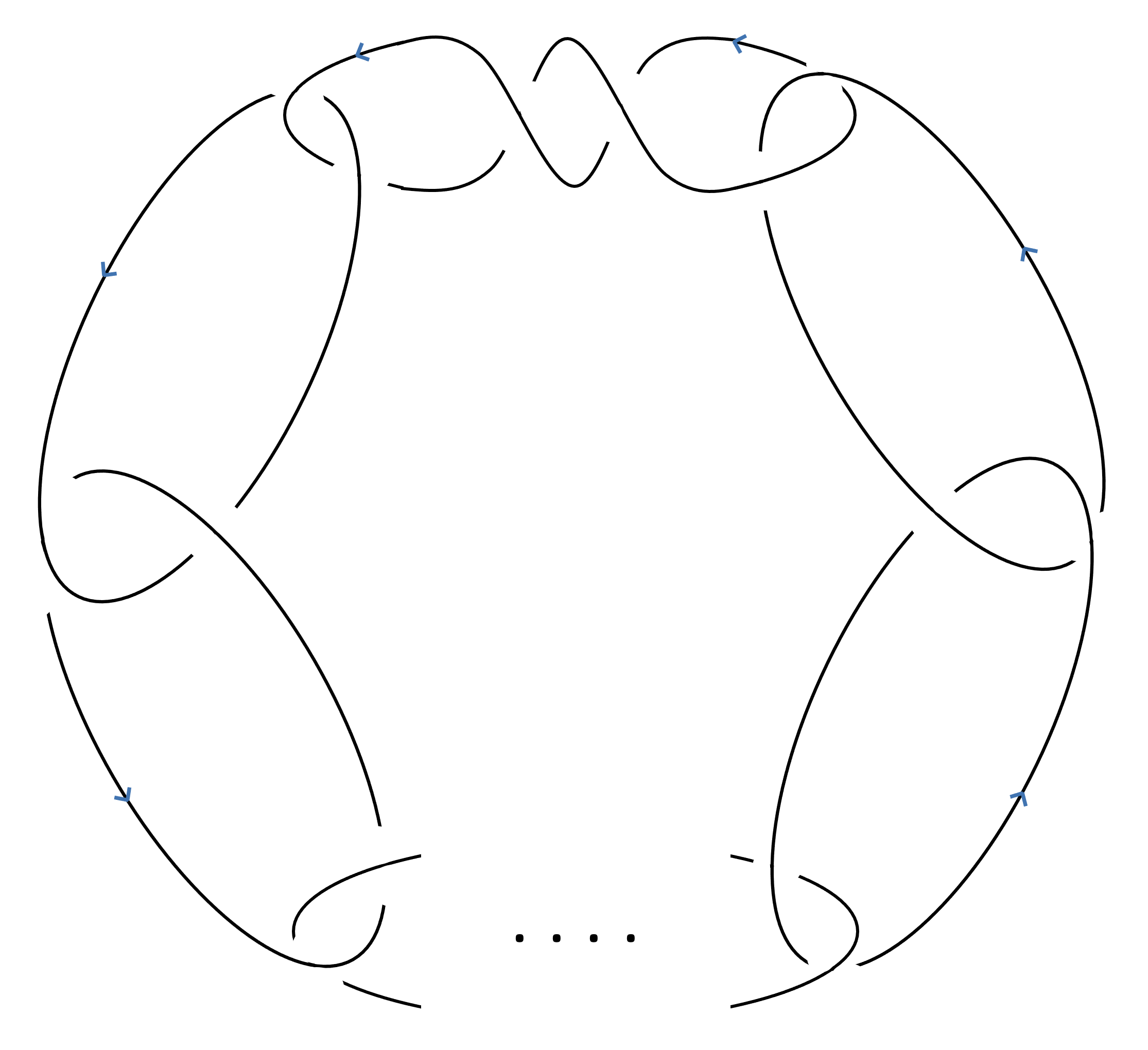}
    \caption{The link $C(n)$}
    \label{fig:C(n)}
\end{figure}
We fist define the family of links we will work with.
Recall that an alternating link is a link which admits a diagram which is alternating.
\begin{defn}[$n$ chained link $C(n)$]
	$C(n)$ is a link composed of $n$ components which are circularly linked together with claps of the same type, and with $2$ half twists on one of its links.
	We choose the direction of these half twists so that $C(n)$ is non-alternating.
	\emph{i.e.,} $C(n)$ is not an alternating link.
	See figure \ref{fig:C(n)}.
\end{defn}

\begin{rmk}
    To clarify the definition of a clasp, we borrow the wording of Leininger in  \cite{leininger2002surgeries}. 
    A clasp is defined to be a pair of crossings where the two ends of adjacent components are linked. 
    \begin{figure}[h]
    	\centering
    	\includegraphics[scale = 0.4]{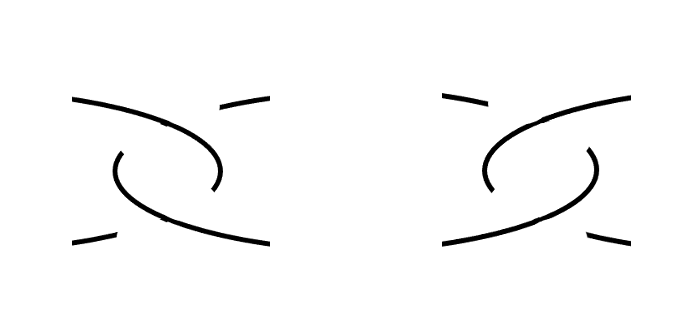}
    	\caption{The two different types of clasps}
    	\label{fig:clasps}
    \end{figure}
	There are exactly $2$ kinds of clasps, as illustrated in figure \ref{fig:clasps}.
	In our case, we will require $C(n)$ to have all clasps looking like the right clasp of figure \ref{fig:clasps}.
	We also give an orientation on the clasp, and it induces the orientation on each component in $C(n)$.
	Note that the half twists can be resolved by an appropriate isotopy, but clasps might change their shape in the process.
	See figure \ref{fig:isotopy_of_twist}. % need a figure of C(n,-2), and half twist resolving isotopy.
\end{rmk}
\begin{figure}[h]
    \centering
    \includegraphics[width = 0.7\textwidth]{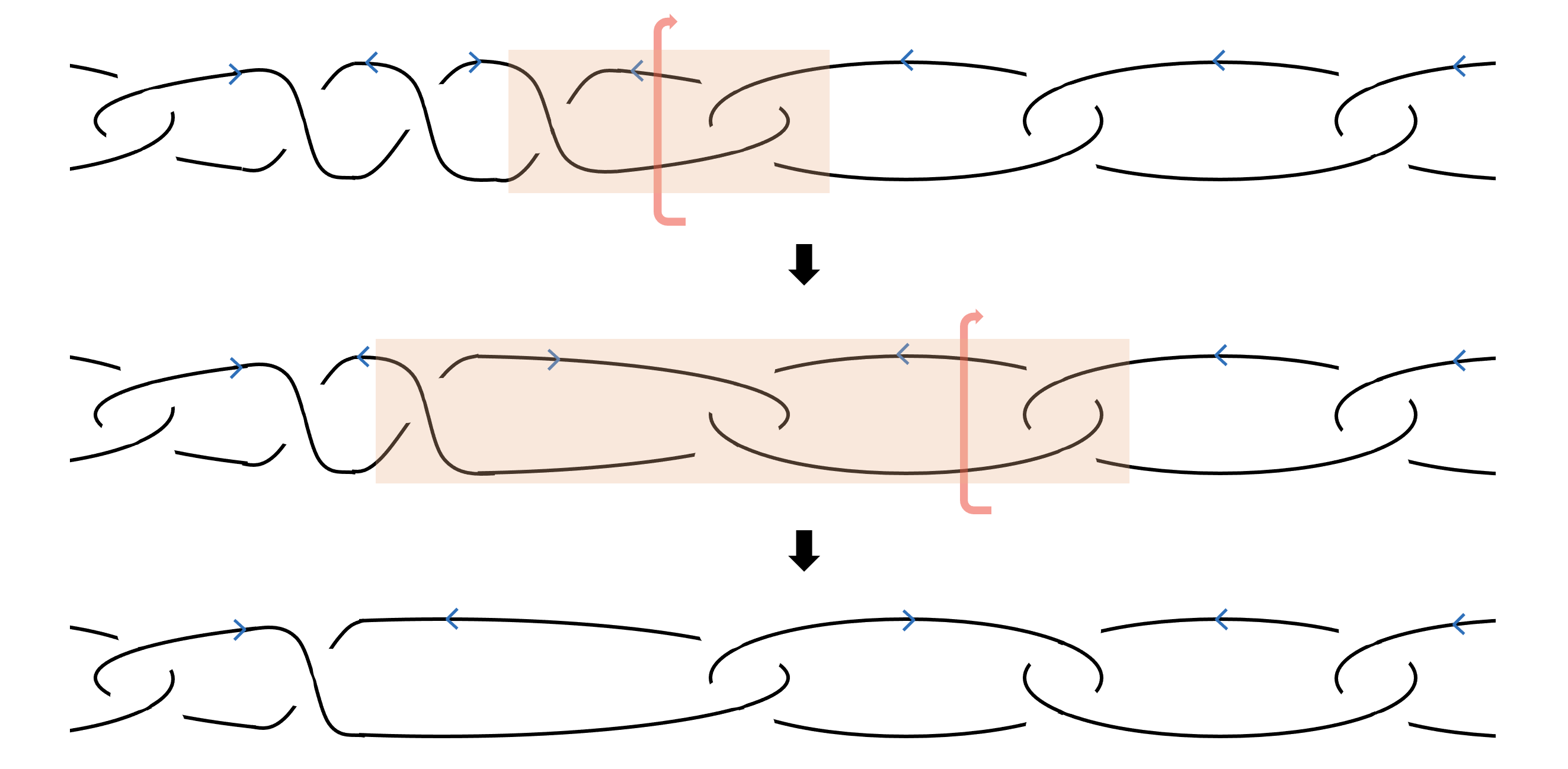}
    \caption{Half twists can be resolved, but the shape of some clasps might change.}
    \label{fig:isotopy_of_twist}
\end{figure}

Let $M(n)$ be the complement of a small enough neighborhood $\mathcal{N}(C(n))$ of $C(n)$.
Note that $M(n)$ is a 3-manifold with boundary and $\partial M = \partial \mathcal{N}(C(n))$ is a disjoint union of $n$ tori.
Also, since the first Betti number of $M$ is $n$, every monodromy map of its fiber fixes a homological subspace of dimension $n-1$.
In \cite{neumann2011arithmetic}, Neumann and Reid prove that, when $n \geq 4$, $M(n)$  is hyperbolic. Leininger \cite{leininger2002surgeries} then showed that it is fibered, by finding a specific fiber.
	
\begin{lem}[Lemma 4.2 in \cite{leininger2002surgeries}]\label{lem:Leininger_fiber}
	$C(n)$ is a fibered link and thus $M(n)$ is a fibered $3$-manifold.
\end{lem}
We want to find a precise fibration of $M$ over the circle, which will then constitute a proof of Lemma \ref{lem:Leininger_fiber}. We need a few more tools before doing so.
First, we recall the definition of an operation on surfaces, called the Murasugi sum.
\begin{defn}[Murasugi sum, \cite{gabai1985murasugi}]
	The oriented surface $\Sigma \subset S^3$ is a Murasugi sum of two different oriented surfaces $\Sigma_1$ and $\Sigma_2$ if 
	\begin{enumerate}
		\item $\Sigma = \Sigma_1 \cup \Sigma_2$ and $\Sigma_1 \cap \Sigma_2 = D$, where $D$ is a $2n$-gon.
		\item There is a partition of $S^3$ into two $3$-balls $B_1, B_2$ satisfying 
		\begin{itemize}
			\item $\Sigma_i \subset B_i$ for $i = 1,2$.
			\item $B_1\cap B_2 = S^2$ and $\Sigma_i \cap S^2 = D$ for $i = 1,2$.
		\end{itemize}
	\end{enumerate} 
\end{defn}
This definition is most likely easier understood by a picture, so we refer to figure \ref{fig:murasugi_sum} for an example of a Murasugi sum.

\begin{figure}
    \centering
    \includegraphics[width = 0.8\textwidth]{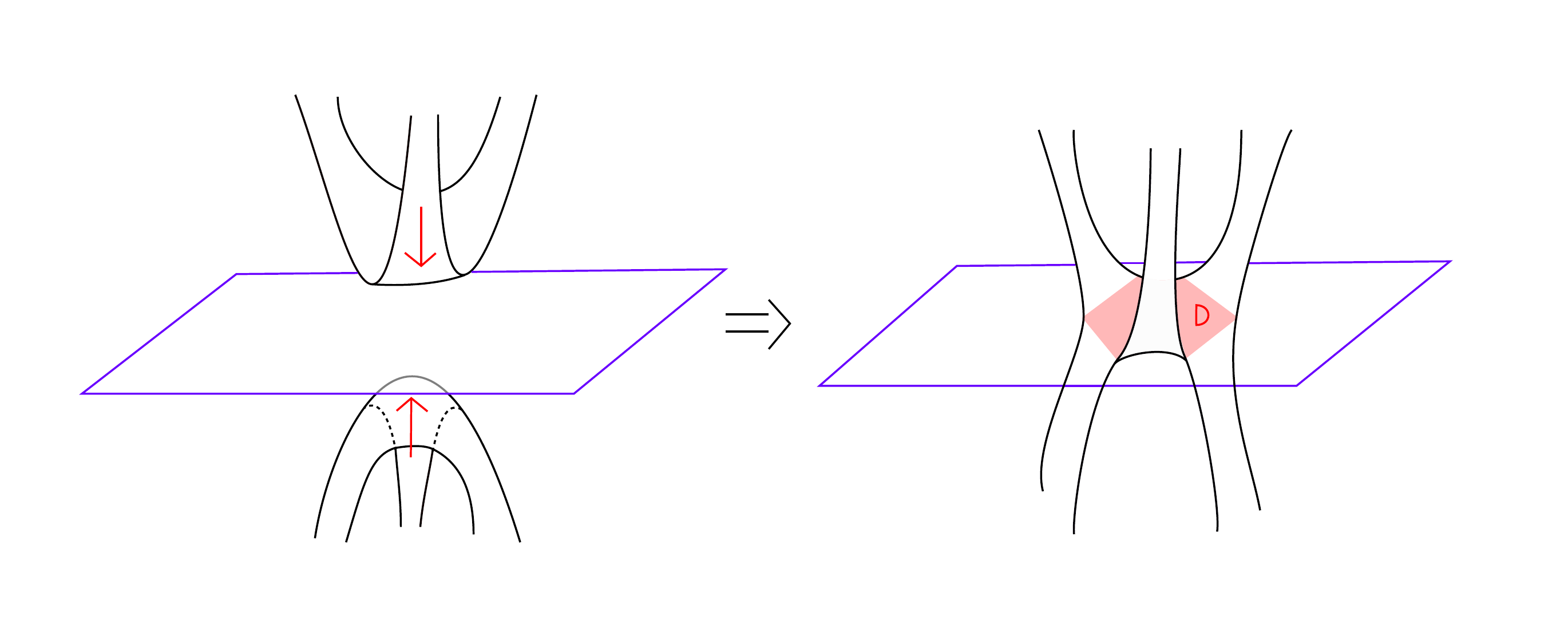}
    \caption{Murasugi sum of two surfaces, here $D$ is a hexagon}
    \label{fig:murasugi_sum}
\end{figure}

By the following theorem of Gabai (\cite{gabai1983murasugi}) we can detect fibers by showing that they are built from smaller fibers.

\begin{thm}[Gabai]\label{thm:murasugi_sum}
	Let $S$ be a Murasugi sum of $S_1$ and $S_2$.
	$S$ is a fiber surface if and only if both $S_1$ and $S_2$ are fiber surfaces.
\end{thm}
We remark that in the above theorem, not only is $S$ a fiber, but we can also construct the associated monodromy from the monodromies of $S_1$ and $S_2$.
In our case, the building block for making more complicated fibers using this process will be Hopf bands.
\begin{lem}[Hopf band is a fiber]\label{lem:hopf_band}
	A Hopf band is a fibered surface.
	Moreover, the monodromy of a positive (negative) Hopf band is the right-handed (left-handed) Dehn twists along its core curve.
\end{lem}

We are now ready to recall the proof of the lemma \ref{lem:Leininger_fiber}.
\begin{proof}
	We will directly prove the lemma by constructing an explicit fibration with fiber $S$.
	The fiber $S$ is obtained from the Seifert algorithm applied to $C(n)$ equipped with a suitable orientation.
	See also the figure \ref{fig:FiberS}.
	
	\begin{figure}
        \centering
        \includegraphics[width = 0.7\textwidth]{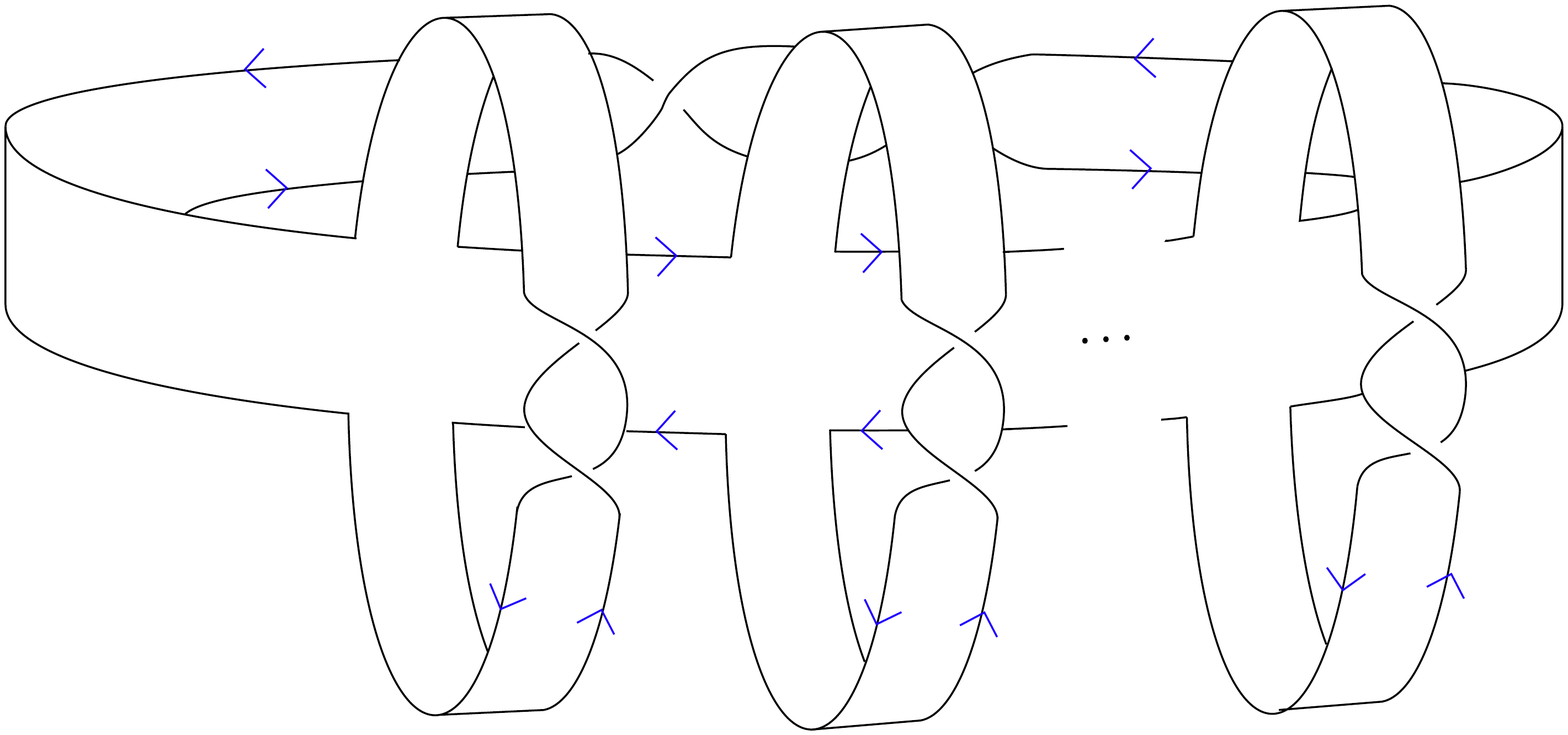}
        \caption{The fiber $S$ is a consecutive Murasugi sum of $n$ vertical bands on a horizontal band.}
        \label{fig:FiberS}
    \end{figure}
    
	Since the Seifert surface $S$ we obtain in this way is a consecutive Murasugi sum of one horizontal Hopf band and $n$ vertical Hopf bands, it is a fiber by Theorem \ref{thm:murasugi_sum} and Lemma \ref{lem:hopf_band}.
\end{proof}

Now, we focus on the homology of $M := M(n)$.
Consider that we draw $C(n)$ in such a way that the top link that has $2$ half-twists, as in the figure \ref{fig:C(n)}.
We then label the top link by $L_1$ and we enumerate the other link component in a clockwise fashion. 
\emph{i.e.,} The links components are labeled $L_1, \cdots, L_n$ in $C(n)$ with $2$ half twists in $L_1$.

Denote the twice punctured disk surrounded by the $i$'th link component by $K_i$. 
Note that the set $[K_i]_{1\leq i \leq n}$ forms a basis of $H_2(M, \partial M)$.
Then we have the following lemma.
\begin{lem}\label{lem:H2coord}
	The fiber $S$ has a coordinate $(1, \cdots, 1)$ with respect to the basis $[K_i]_{1\leq i \leq n}$.
	\emph{i.e.,} $[S] = [K_1] + \cdots + [K_n]$.
\end{lem}
\begin{proof}
	The boundary map $\partial_* : H_2(M, \partial M) \to H_1(\partial M)$ sends $(1,\cdots,1)$  to the boundary of the Seifert surface.
	Therefore it suffices to show that $\partial_*$ is injective.
	
	In the long exact sequence for $(M, \partial M)$, the map right after the boundary map is induced by the inclusion $i : \partial M \to M$.
	Only the meridians survive under $i_*$ and they form a basis of $H_1(M)$.
	Hence $i_*$ is surjective.
	Now we will cut out the following short exact sequence from the long exact sequence, 
	\[
	0\to H_2(M, \partial M) / \text{image of } j_* \to H_1(\partial M) \to H_1(M) \to 0
	\]
	where $j : (M, *) \hookrightarrow (M, \partial M)$ is an inclusion.
	Since $H_1(M) \cong \mathbb{Z}^n$, it splits.
	Thus $\mathbb{Z}^{2n} \cong H_1(\partial M) = H_1(M) \oplus H_2(M, \partial M) / \text{image of } j_* \cong \mathbb{Z}^n \oplus \mathbb{Z}^n / \text{image of } j_*$.
	Therefore, $j_*$ is a zero map.
	In conclusion, $\partial_*$ is injective and the fiber $S$ has coordinates $(1 ,\cdots, 1)$ in the basis $[K_i]_{1\leq i \leq n}$.
\end{proof}
Note that the fiber $S$ is a genus $1$ surface with $n$ boundaries and so its Euler characteristic is equal to $n$.
From now on, let $\mathcal{F}$ be the fibered face of $M := M(n)$ which contains $S$.

We first observe the following lemma.
\begin{lem}\label{lem:eqn_of_fib_face}
	Let $M = M(n)$ and suppose that $a_1, \cdots, a_n$ are $n$ linearly independent points in $H_2(M,\partial M)$, all of which have Thurston norm equal to $1$.
	Let $\sigma$ be the $(n-1)$-dimensional simplex obtained by taking the convex hull of $a_1, \cdots, a_n$. Then, if there exists a point $a$ in interior of $\sigma$ whose Thurston norm $x(a)$ is equal to $1$, $\sigma$ is a subset of a fibered face $\mathcal{F}$ of $M$.
	Moreover the hyperplane supporting the fibered face $\mathcal{F}$ is simply the hyperplane passing through $a_1, \cdots, a_n$.
\end{lem}
\begin{proof}
	The proof is a direct consequence of the convexity of the unit Thurston's norm ball.
\end{proof}

Then, as in the magic $3$-manifold case, we can calculate the Euler characteristic of any primitive points in $\mathcal{F}$.
\begin{cor}\label{cor:sub_facets_in_C(n,0)}
	The convex hull of the points $e_1,e_2,\cdots,e_n$ and $\dfrac{1}{n}(1,\cdots,1)$ is a subset of the fibered face $\mathcal{F}$.
	Moreover, for any primitive point $\alpha := (\alpha_1,\cdots,\alpha_n)$ in the cone $\mathcal{C} := \R^+\cdot\mathcal{F}$, the Euler characteristic of the representative of $\alpha$ is $\alpha_1+\cdots+\alpha_{n}$.
\end{cor}
\begin{proof}
	Use Lemma \ref{lem:eqn_of_fib_face} with $a_i = e_i$ and $a = \dfrac{1}{n}(1,\cdots,1)$.
	Since we already observed that $na$ is a fiber and $-\chi(na) = n$, it gives a linear equation $x_1 + \cdots + x_n = 1$, which is then the equation of a supporting hyperplane for the fibered face $\mathcal{F}$.
	Plugging $(\alpha_1, \cdots, \alpha_n)$ into $x_1 + \cdots + x_n$, we get the Euler characteristic for $\alpha$.
\end{proof} 

To understand fully the topological type of surfaces representing a given fibered point in $\mathcal{C}$, we will use a slightly generalized version of the boundary formula proven by Kin and Takasawa, \cite{kin2008pseudo}.
\begin{lem}[slight generalization of Lemma 3.1 in \cite{kin2008pseudo}]\label{lem:gen_bdd_formula}
	Suppose $\Sigma$ is a minimal representative of $(a_1, \cdots, a_n) \in \mathcal{C}$.
	Then the number of boundaries of $\Sigma$ is equal to $\sum_{i = 1}^n \gcd(a_{i-1} + a_{i+1}, a_i)$, where the subscripts are understood modulo $n$.
\end{lem}
\begin{proof}
	The proof is analogous to the original one in \cite{kin2008pseudo}.
	Note that $\partial \mathcal{N}(L_i) := \bigcup_{i = 1}^n T_i$, where $T_i$ is a torus whose longitude is $L_i$.
	So we can talk about the longitudes $[l_i]$'s and meridians $[m_i]$'s of each $T_i$, which together form a basis of $H_1(\bigcup_{i = 1}^n T_i)$.
	Then using the long exact sequence for the pair $\left(M(n), \partial M(n) = \bigcup_{i = 1}^n T_i)\right)$, we get a boundary map $\partial_*$ as in the lemma \ref{lem:H2coord}.
	
	Note that $\partial_*$ sends $[K_i] \mapsto l_i - m_{i-1} - m_{i+1}$, where the subscript $i$ is to be understood modulo $n$.
	Hence,
	\[
	\sum_{i = 1}^n a_i[K_i] \mapsto \sum_{i = 1}^n a_il_i - \sum_{i = 1}^n(a_i m_{i-1} + a_i m_{i+1})
	\]
	Since $\Sigma$ is the minimal representative of $(a_1, \cdots, a_n)$, the boundary of $\Sigma$ which intersects with $T_i$ is a union of oriented parallel simple closed curves on $T_i$.
	The slope of such a curve is decided by the ratio of $l_i, m_i$, which is $(a_{i-1}+a_{i+1}, a_i)$.
	Similarly, the number of boundaries that intersects with $T_i$ is equal to $\gcd(a_{i-1}+a_{i+1}, a_i)$.

\end{proof}

\subsection{Monodromy of $S$}\label{sub:monodromy}
Now we focus on the monodromy map with fiber $S$, provided from \cite{leininger2002surgeries}.
We need another theorem proven by Gabai.
\begin{thm}[\cite{gabai1985murasugi}, Cor 1.4]\label{thm:monodromy}
	Suppose that $R$ is a Murasugi sum of $R_1, R_2$ with $\partial R_i = L_i$, where $L_i$ is a fibered link with monodromy $f_i$ which restricts to the identity on $\partial R_i$, resp.
	Then $L = L_1\cup L_2$ is fibered link with the fiber $R$ and its monodromy map is $f = f_2'\circ f_1$ where $f_i'$ is the induced map on $R$ by inclusion.
\end{thm}

By Gabai's Theorem \ref{thm:monodromy} and Lemma \ref{lem:hopf_band}, the monodromy of our given fiber $S$ is equal to the composition of the Dehn twists around the Hopf bands.
More precisely, $S$ admits a monodromy composed of $1$ horizontal Dehn twist followed by $n$ vertical Dehn twist.
Thus we have the following corollary,
\begin{cor}[Monodromy fixes punctures]\label{cor:fixdim}
	Let $\psi$ be the monodromy map of given fiber $S$.
	Then $\psi_* : H_1(S) \to H_1(S)$ fixes a subspace of dimension at least $n-1$.
\end{cor}

We end this subsection with one remark. 
In \cite{leininger2002surgeries}, Leininger proved not only that the $n$-chained link that has only $2$ half twists is fibered, but also that $n$-chained links with $p$ half twists with $0 \leq p \leq n$ are fibered, except when $(n,p) = (2,-1)$.
Since such fibers are still a consecutive Murasugi sum of Hopf bands, the monodromy of each fiber is easily understood.
In a future paper, the second and fourth author intend to study the unit Thurston norm ball, the fibered faces and the Teichm\"uller polynomials of $n$-chained links complements with $p$ half twists.

\subsection{Stretch factors for the ray $(1, \cdots, 1)$}\label{sec:stretch_factor}
In this section we compute the stretch factor of given fiber in $\mathcal{F}$.
We denote the surface obtained from performing the Seifert algorithm to $C(n)$ by $S := S_n$.
Since $M(n)$ is the complement of $C(n)$, the second homology $H_2 = H_2(M(n), \partial M(n); \mathbb{Z})$ is a free abelian group of rank $n$, with a canonical basis given by the meridians of the link component.
With that in mind, we remark that $S_n$ is a surface of genus one with $n$ boundaries and its coordinates in $H_2$ are $(1,1,\cdots,1)$.

Thus, if $S_n$ is placed as suggested in figure \ref{fig:FiberS}, the monodromy $\psi_n$ is the composition of the $n$ vertical multi-twists directed downward followed by the left Dehn twist along the core of horizontal band.

\begin{prop}
   The stretch factor of the monodromy corresponding to $(1,1,\cdots,1)$ is $\dfrac{n+2 +\sqrt{n^2 + 4n}}{2}$.
\end{prop}
\begin{proof}
Recall from section \ref{subsec:Thurstonconstruction}, Thurston constructed a method to make pseudo-Anosov element from pair of filling multicurves and its Dehn tiwsts. The core of the horizontal band and the vertical bands fill the surface by dividing it into once punctured disks, therefore our monodromy is obtained from Thurston's construction. We use the notations of Theorem \ref{Thurston's construction}.
Since the core of the horizontal band meets each vertical bands once, we have that $N=(1,\cdots,1)$. 
Hence, $NN^T=(n)$ and thus $\mu =n$. 
The map $T_AT_B^{-1}$ represents the monodromy and the trace of its representation is equal to $n+2$. 
It is thus hyperbolic and the stretch factor of the monodromy is $\dfrac{n+2 +\sqrt{n^2 + 4n}}{2}$.  
\end{proof}

Note that the stretch factor is asymptotic to $n$ and smaller than $n+2$ here. Therefore the entropy is asymptotic to $\log n$ and strictly less than $\log(n+2)$. Remark that in the case of this monodromy, an invariant train track is easily found by smoothing the cores of the Hopf bands properly. Since this invariant train track provides an oriented singular foliations on $S_n$, the stretch factor can be calculated from the leading eigenvalue of its homology action.
Also, one could compute the stretch factors from the leading eigenvalue of the transition matrix of the invariant train track.

\section{Sequences around the ray $(1, \cdots, 1)$}\label{sec:sequence}
In this section, we find many sequences of fibers in $\mathcal{F}$ that projectively converge to the ray $R$ passing through the point $\mathbbm{1}_N := (\underbrace{1, \cdots, 1}_{N \text{ times}})$.
As before, we denote by $M := M_N$ the mapping torus of $S_{1,N}\times[0,1] / (x, 0) \sim (\psi(x),1)$, where $\psi := \psi_N$ is the pseudo-Anosov map discussed in Section \ref{sub:monodromy}.
We will drop the index $N$ when no confusion is possible.
Even though we do not know the exact shape of the fibered cone $\mathcal{C}$ which contains the point $\mathbbm{1}$, recall that from section \ref{sec:nchainlink} we know the following.
\begin{prop} \label{prop:simplex}% Maybe relocated to one of the above sections
	Let $E = \{e_1,\cdots,e_n\}$ be the standard basis for $\R^n$ and let $\sigma$ be the simplex spanned by $E$.
	Then the interior of the cone over $\sigma$ is a subset of the fibered cone $\mathcal{C}$ of $M$ containing the point $\mathbbm{1}_N$.
	Moreover, the Euler characteristic for any primitive point $(x_1, \cdots, x_N)$ in $\mathcal{C}$ is given by $\sum_{i = 1}^N x_i$.
\end{prop}

We thus know that not only the ray $R$ but all points with positive integer coordinates are in the interior of the fibered cone $\mathcal{C}$, and we can compute the entropy of points converging to that ray. We consider the following sequences of points, all of which have $N \geq 4$ coordinates,
\hfill
\begin{center}
	\begin{tabular}{ c c c }
		$(t+1, t+1, t, t, \cdots, t)$ &  & $(N-4)t + 4$ \\ 
		$(t+1, t+1, t+1, t, t, \cdots, t)$ &  & $(N-4)t + 5$ \\  
		$\vdots$ &  & $\vdots$ \\
		$(t+1, t+1, \cdots, t+1, t, t)$ & & $(N-4)t + N$
	\end{tabular}
\end{center}
The numbers appearing on the right of the sequences are the number of boundary components of the associated fibers, which are calculated using Lemma \ref{lem:gen_bdd_formula}.

Similarly, we can compute the number of genera $g$ for each of these sequences by using an Euler characteristic argument:
\[
    Nt + c = 2g + (N-4)t + c + 2 - 2  \qquad \text{ where } 2\leq c\leq N-2.
\]
Hence, the number of genera is the same for each sequences and equal to $g = 2t$.

Similarly, we consider following sequences of points having $N \geq 8$ coordinates.
\begin{center}
	\begin{tabular}{ c c c }
		$(t+1, t+1, t, t, t+1, t+1, t, t, \cdots, t)$ &  & $(N-8)t + 8$ \\ 
		$(t+1, t+1, t+1, t, t, t+1, t+1, t, t, \cdots, t)$ &  & $(N-8)t + 9$ \\  
		$\vdots$ &  & $\vdots$ \\
		$(t+1, t+1, \cdots, t+1, t+1, t, t, \cdots, t)$ & & $(N-8)t + N$
	\end{tabular}
\end{center}
The number of boundaries is again written on the right of the sequences and the number of genera is $g = 4t-1$ as shown by the following calculation:
\[
Nt + c = 2g + (N-8)t + (c+4) -2 \qquad \text{ where } 4\leq c\leq N-4, \quad \Rightarrow \quad g = 4t-1
\]

Inductively, we have infinitely many of these finite sequences which starts with $n$ number of $(t+1)$'s followed by $m$ number of $(t+1, t+1, t, t)$'s , and end with some consecutive $t$'s.
The table below summarize the conditions and the number of genera and boundaries obtained from each of these sequences.

~

\begin{center}
	\begin{tabular}{ c c c }
		\#genera & \#boundaries & condition on $N$\\
		$2t$ & $(N-4)t + c, ~4\leq c \leq N$ & $N\geq 4$\\
		$4t-1$ & $(N-8)t + c, ~8\leq c \leq N$ & $N\geq 8$\\
		$6t-2$ & $(N-12)t + c, ~12\leq c \leq N$ & $N\geq 12$\\
		$\vdots$ & $\vdots$ & $\vdots$ \\
		$2mt - (m-1)$ & $(N-4m)t + c, ~4m\leq c \leq N$ & $N\geq 4m$ \\
		$\vdots$ & $\vdots$ & $\vdots$ 
	\end{tabular}
\end{center}

Take $t=1$ and choose the first term on each sequence so that each tuple represents $S_{m+1,N}$.
This sequence starts with 
\[
m = 1 \quad \Rightarrow \quad (2,2,1,1,\underbrace{1,\cdots,1}_{i \text{ times}})
\]
\[
m = 2 \quad \Rightarrow \quad (2,2,1,1,2,2,1,1,\underbrace{1,\cdots,1}_{i \text{ times}}) 
\]
\[
m = 3 \quad \Rightarrow \quad(2,2,1,1,2,2,1,1,2,2,1,1,\underbrace{1,\cdots,1}_{i \text{ times}})
\]
and so forth. 

Also, if we choose the second term on each sequence, then each tuple represents $S_{m+1, N+1}$.
This sequence starts with
\[
m = 1 \quad \Rightarrow \quad (2,2,2,1,1,\underbrace{1,\cdots,1}_{i \text{ times}})
\]
\[
m = 2 \quad \Rightarrow \quad (2,2,2,1,1,2,2,1,1,\underbrace{1,\cdots,1}_{i \text{ times}}) 
\]
\[
m = 3 \quad \Rightarrow \quad(2,2,2,1,1,2,2,1,1,2,2,1,1,\underbrace{1,\cdots,1}_{i \text{ times}})
\]
and so forth. 

\begin{defn}[$(m,n)$-sequence]
	We call the sequence above an $(m,n)$-sequence and denote it by $\{s^{(m,n)}_i\}_{i\in \N}$. 
	\emph{i.e.,} 
	\[
	s^{(m,n)}_i = (\underbrace{2, \cdots, 2}_{n \text{ times}},\underbrace{2,2,1,1,\cdots,2,2,1,1}_{m \text{ number of } (2,2,1,1)\text{'s}},\underbrace{1,\cdots,1}_{i \text{ times}})
	\]
	For example, the $(1,1)$-sequence starts with $s^{(1,1)}_1 = (2,2,2,1,1,1)$, $s^{(1,1)}_2 = (2,2,2,1,1,1,1)$, $s^{(1,1)}_3 = (2,2,2,1,1,1,1,1)$ and so on.
	Note that the length of $s^{(m,n)}_i$ is $N := n+4m+i$ and each tuple represents a fiber of $M_N$ of topological type $S_{m+1, N+n}$.
\end{defn}

The following lemma is useful to approximate the upper bounds of entropies of $s^{(m,n)}_i$.
\begin{lem}[Lemma 3.11 in \cite{agol2016pseudo}]\label{lem:nor_ent_convex}
	Let $\mathcal{C}$ be a fibered cone for a mapping torus $M$.
	Then the entropy function $\mathfrak{h}(x)$ for $x\in \mathcal{C}$ is strictly convex.
	\emph{i.e.,} $\mathfrak{h}(u+v) < \mathfrak{h}(u)$ for $u\in \mathcal{C}$ and $v \in \overline{\mathcal{C}}$.
\end{lem}

Applying this lemma to our $(m,n)$-sequences, we get the following
\begin{lem}\label{lem:nor_ent_bdd}
	The normalized entropy $\overline{\mathfrak{h}}(s^{(m,n)}_i)$ is less than $2N\log{(N+2)}$, where $N = n+4m+i$, the length of $s^{(m,n)}_i$.
\end{lem}
\begin{proof}
	We apply lemma \ref{lem:nor_ent_convex} with 
	\[
	u = \mathbbm{1}_{n + 4m + i} \quad \text{and} \quad v = (\underbrace{1, \cdots, 1}_{n \text{ times}},\underbrace{1,1,0,0,\cdots,1,1,0,0}_{m \text{ number of } (1,1,0,0)\text{'s}},\underbrace{0,\cdots,0}_{i \text{ times}})
	\]
	Note that $v$ is in the boundary of the standard simplex $\sigma$, so that $v\in \overline{\mathcal{C}}$ by Lemma \ref{prop:simplex}.
	Since $s^{(m,n)}_i = u+v$, the entropy of $s^{(m,n)}_i$ is strictly less than the entropy of $\mathbbm{1}_N$, which is in turn less than $\log{(N+2)}$.
	As $s^{(m,n)}_i$ represents a surface $S_{m+1, N+n}$, multiplying by $|\chi(S_{m+1, N+n})|$ gives the inequality
	\[
	    \overline{\mathfrak{h}}(s^{(m,n)}_i) < \dfrac{|\chi(S_{m+1, N+n})|}{N}\overline{\mathfrak{h}}(\mathbbm{1}_N) < 2N\log{(N+2)}
	\]
	The last inequality comes from $N=4m+n+i$, so that $\dfrac{|\chi(S_{m+1, N+n})|}{N}=\dfrac{6m+2n+i}{4m+n+i}<2$. Note that $\dfrac{|\chi(S_{m+1, N+n})|}{N}$ monotonely decreases and converges to $1$ as $i \to \infty$.
\end{proof}

Now we are ready to prove the upper bound for various $L(k, g, n)$.
\begin{thm}\label{thm:main_ineq}
	For $4g-4 \leq k+1 \leq n \leq 2k-4g+6$, we have
	\[
	L(k, g, n)\leq \frac{2(k+1)\log{(k+3)}}{|\chi(S_{g, n})|} 
	\]
\end{thm}
\begin{proof}
	Given $(k,g,n)$, we choose the tuple $\mathfrak{s} := s^{(g-1, n-k-1)}_{2k - 4g + 6 - n}$.
	Note that the length of this tuple is $N := 4(g-1) + n-k-1 + 2k - 4g + 6 - n = k+1$ and it represents a fiber of topological type $S_{g, n}$.
	By Lemma \ref{lem:nor_ent_convex}, dividing the normalized entropy by $|\chi(S_{g,n})|$, we get an upper bound on entropy as follows,
	\[
	\mathfrak{h}(\mathfrak{s}) < \frac{2(k+1)\log{(k+3)}}{|\chi(S_{g,n})|}
	\]
	
	Finally, all surfaces we have considered are essentially embedded fibers of the mapping torus $M := S_{1,N}\times [0,1] / (x, 0)\sim (\psi(x),1)$. Hence, the first Betti number of $M$ determines the dimension of fixed homological subspace of the action of $\psi$ on $H_1$, which is then $N-1 = k$. 
	Therefore, we did indeed find an upper bound for $L(k, g, n)$ with conditions on $(k,g,n)$ as in the statement of the theorem.
\end{proof}
Restricting the scope of $(k,g,n)$ in Theorem \ref{thm:main_ineq}, we get a few direct corollaries.
If we take $k = n-1$ we have the following,
\begin{cor}
	For $n \geq 4g-4$, we have
	\[
	L(n-1, g, n)\leq \frac{2n\log{(n+2)}}{|\chi(S_{g, n})|} 
	\]
\end{cor}
If instead we fix $g \geq 2$, we can see that $L(k,g,n)$ has, up to constant, an upper bound of the form $\dfrac{k\log{(n+2)}}{|\chi(S_{g,n})|}$.
\begin{cor}
	Let $g \geq 2$ fixed and $n \geq 4g-4$.
	For $\dfrac{n+4g-6}{2} \leq k < n$,
	\[
	    L(k,g,n) \leq \frac{2(k+1)\log{(n+2)}}{|\chi(S_{g,n})|}
	\]
\end{cor}
\begin{proof}
	The condition $k+1 \leq n \leq 2k - 4g +6$ from Theorem \ref{thm:main_ineq} can be reduced to $\dfrac{n+4g-6}{2} \leq k < n$.
\end{proof}

We end this paper with a few questions we believe are worth thinking about.
Note that in \cite{tsai2009asymptotic}, Tsai proved that for any fixed $g \leq 2$, the minimal entropy $l_{g,n}$ of pA map on $S_{g,n}$ satisfies the following inequality,
\[
\frac{\log{n}}{c_g n} < l_{g,n} < \frac{c_g \log{n}}{n}
\]
for a constant $c_g$ only depends to $g$.
This can be restated as `$L(0,g,n)$ behaves like $\log{n}/n$ asymptotically'.
Our upper bound is not tight enough to describe the asymptotic behavior of $L(k,g,n)$. This is perhaps due to the absence of nicer lower bounds at the moment, or perhaps due to the existence of other sequences with lower entropy. 
It is thus natural to ask, as we did in Question \ref{que:main}, whether $L(k,g,n)$ behaves asymptotically like $\dfrac{(k+1)\log n}{|\chi(S_{g,n})|}$.

Note that in our formula we have a  $k\log{k}$ in the numerator, but the range of $k$ is restricted linearly by $n$ and that is why we believe the general formula should feature a $\log{n}$ and not $\log{k}$.

Another natural improvement would be to relax the conditions on $k, g$ and  $n$ in our theorem. 
Also, in the condition on the inequality, $k \geq 4g - 5$ implies $k \geq 5$ as we only consider surfaces of genera $g\geq 2$. 

We remark that in chapter \ref{sec:magic3mfld} we found some sequences that fix a homological subspace of dimension $2$, but we could not generalize the techniques used there for finding sequences fixing homological subspaces of arbitrary dimensions. That would also be a interesting directions to pursue.

Moreover, we only use the case $t = 1$ for the $(m,n)$-sequences we defined, but one can plug different values of $t$'s to get many other sequences.
But these sequences have gaps on the number of genus. For example, for $t = 2$, the genera covers by the sequences will be $g = 4, 7, 10,$ and so on.
Still, studying these sequences further might be useful to get bounds of $L(k,g,n)$ with a different scope on $k$, since enlarging $t$ does not change $k$ but increases $g$ only or both $g$ and $n$.
If we fix the length of tuple as $N = k+1$ and increasing $t$, then the upper bound of $L(k,g,n)$ is asymptotic to $1/{|\chi(S_{g,n})|}$.

\bibliographystyle{alpha} 
\bibliography{topology}

\begin{thebibliography}{BKSW19}

\bibitem[ALM16]{agol2016pseudo}
Ian Agol, Christopher~J Leininger, and Dan Margalit.
\newblock Pseudo-anosov stretch factors and homology of mapping tori.
\newblock {\em Journal of the London Mathematical Society}, 93(3):664--682,
  2016.

\bibitem[BKSW19]{baik2019asymptotic}
Hyungryul Baik, Eiko Kin, Hyunshik Shin, and Chenxi Wu.
\newblock Asymptotic translation lengths and normal generation for
  pseudo-anosov monodromies of fibered 3-manifolds.
\newblock {\em arXiv preprint arXiv:1909.00974}, 2019.

\bibitem[CTW21]{cooper2021thurston}
Daryl Cooper, Stephan Tillmann, and William Worden.
\newblock The thurston norm via spun-normal immersions.
\newblock {\em arXiv preprint arXiv:2109.04498}, 2021.

\bibitem[Ell10]{ellenberg2010pseudo}
Jordan Ellenberg.
\newblock Pseudo-anosov puzzle 2: homology rank and dilatation, quomodocumque,
  2010.

\bibitem[FLM08]{farb2008lower}
Benson Farb, Christopher~J Leininger, and Dan Margalit.
\newblock The lower central series and pseudo-anosov dilatations.
\newblock {\em American journal of mathematics}, 130(3):799--827, 2008.

\bibitem[FLP79]{FLP}
Albert Fathi, Fran\c{c}ois Laudenbach, and Valentin Po\'{e}naru.
\newblock {\em Travaux de {T}hurston sur les surfaces}, volume~66 of {\em
  Ast\'erisque}.
\newblock Soci\'et\'e Math\'ematique de France, Paris, 1979.
\newblock S{\'e}minaire Orsay, With an English summary.

\bibitem[FM12]{farb2012primer}
Benson Farb and Dan Margalit.
\newblock {\em A primer on mapping class groups}, volume~49 of {\em Princeton
  Mathematical Series}.
\newblock Princeton University Press, Princeton, NJ, 2012.

\bibitem[Fri82]{fried1982flowequivalence}
David Fried.
\newblock Flow equivalence, hyperbolic systems and a new zeta function for
  flows.
\newblock {\em Commentarii Mathematici Helvetici}, 57:237--259, 1982.

\bibitem[Fri83]{fried1983anosovflow}
David Fried.
\newblock Transitive anosov flows and pseudo-anosov maps.
\newblock {\em Topology}, 22(3):299--303, 1983.

\bibitem[Gab83]{gabai1983murasugi}
David Gabai.
\newblock The murasugi sum is a natural geometric operation.
\newblock {\em Amer. Math. Soc. Contemp. Math.}, 20:131--143, 1983.

\bibitem[Gab85]{gabai1985murasugi}
David Gabai.
\newblock The murasugi sum is a natural geometric operation ii.
\newblock {\em Contemp. Math}, 44:93--100, 1985.

\bibitem[Kin14]{kin2014dynamics}
Eiko Kin.
\newblock Dynamics of the monodromies of the fibrations on the magic
  3-manifold.
\newblock {\em arXiv preprint arXiv:1412.7607}, 2014.

\bibitem[KM18]{kojima2018normalized}
Sadayoshi Kojima and Greg McShane.
\newblock Normalized entropy versus volume for pseudo-anosovs.
\newblock {\em Geometry \& Topology}, 22(4):2403--2426, 2018.

\bibitem[KT08]{kin2008pseudo}
Eiko Kin and Mitsuhiko Takasawa.
\newblock Pseudo-anosov braids with small entropy and the magic 3-manifold.
\newblock {\em arXiv preprint arXiv:0812.4589}, 2008.

\bibitem[Lei02]{leininger2002surgeries}
Christopher~J Leininger.
\newblock Surgeries on one component of the whitehead link are virtually
  fibered.
\newblock {\em Topology}, 41(2):307--320, 2002.

\bibitem[NR11]{neumann2011arithmetic}
Walter~D Neumann and Alan~W Reid.
\newblock Arithmetic of hyperbolic manifolds.
\newblock In {\em Topology'90}, pages 273--310. de Gruyter, 2011.

\bibitem[Pen91]{penner1991bounds}
Robert~C Penner.
\newblock Bounds on least dilatations.
\newblock {\em Proceedings of the American Mathematical Society},
  113(2):443--450, 1991.

\bibitem[Str18]{strenner2018fibrations}
Bal{\'a}zs Strenner.
\newblock Fibrations of 3-manifolds and asymptotic translation length in the
  arc complex.
\newblock {\em arXiv preprint arXiv:1810.07236}, 2018.

\bibitem[Thu86]{thurston1986norm}
William~P Thurston.
\newblock A norm for the homology of 3-manifolds.
\newblock {\em Memoirs of the American Mathematical Society}, 59(339):99--130,
  1986.

\bibitem[Thu88]{thurston1988construction}
William~P. Thurston.
\newblock {On the geometry and dynamics of diffeomorphisms of surfaces}.
\newblock {\em Bulletin (New Series) of the American Mathematical Society},
  19(2):417 -- 431, 1988.

\bibitem[Tsa09]{tsai2009asymptotic}
Chia-Yen Tsai.
\newblock The asymptotic behavior of least pseudo-anosov dilatations.
\newblock {\em Geometry \& Topology}, 13(4):2253--2278, 2009.

\bibitem[Val12]{valdivia2012sequences}
Aaron~D Valdivia.
\newblock Sequences of pseudo-anosov mapping classes and their asymptotic
  behavior.
\newblock {\em New York J. Math}, 18:609--620, 2012.

\bibitem[Yaz20]{yazdi2020punture}
Mehdi Yazdi.
\newblock {Pseudo-Anosov maps with small stretch factors on punctured
  surfaces}.
\newblock {\em Algebraic \& Geometric Topology}, 20(4):2095 -- 2128, 2020.

\end{thebibliography}
	
\end{document}